\documentclass{amsart}

\usepackage{amsmath}
\usepackage{amsthm}
\usepackage{amsfonts}
\usepackage{dsfont}

\newcommand{\bracketb}[1]{\Big[#1\Big]}

\newcommand{\paraa}[1]{\big(#1\big)}
\newcommand{\parab}[1]{\Big(#1\Big)}

\newcommand{\sgn}{\operatorname{sgn}}

\newtheorem{theorem}{Theorem}[section]

\newtheorem{lemma}[theorem]{Lemma}
\newtheorem{proposition}[theorem]{Proposition}
\newtheorem{example}[theorem]{Example}
\theoremstyle{definition}
\newtheorem{definition}[theorem]{Definition}
\theoremstyle{remark}

\numberwithin{equation}{section}

\newcommand{\K}{\mathbb{K}}

\newcommand{\xh}{\hat{x}}
\newcommand{\phit}{\phi_\tau}
\newcommand{\alphatautuple}{(\alpha_1,\ldots,\alpha_n,\tau)}
\newcommand{\A}{\mathcal{A}}
\newcommand{\X}{\mathcal{X}}
\newcommand{\Xt}{\tilde{\X}}
\newcommand{\Y}{\mathcal{Y}}
\newcommand{\Yt}{\tilde{Y}}
\newcommand{\Xb}{\bar{X}}
\newcommand{\id}{\operatorname{id}}
\newcommand{\spanvec}{\operatorname{span}}

\title[]{Construction of $n$-Lie algebras and \\ $n$-ary Hom-Nambu-Lie algebras}

\author{Joakim Arnlind}
\address{Max Planck Institute for Gravitational Physics (AEI), Am M\"uhlenberg 1, D-14476 Golm, Germany.}
\email{arnlind@aei.mpg.de}

\author{Abdenacer Makhlouf}
\address{Universit\'{e} de Haute Alsace,  Laboratoire de Math\'{e}matiques, Informatique et Applications,
4, rue des Fr\`eres Lumi\`ere F-68093 Mulhouse, France}
\email{abdenacer.makhlouf@uha.fr}

\author{Sergei Silvestrov}
\address{Centre for Mathematical Sciences,  Lund University, Box
   118, SE-221 00 Lund, Sweden}
\email{ssilvest@maths.lth.se}

\thanks{
This work was partially supported by The Swedish Foundation for
International Cooperation in Research and Higher Education (STINT), The SIDA Foundation, The
Swedish Research Council, The Royal Swedish Academy of Sciences, The Crafoord Foundation and The Letterstedtska
F{\"o}reningen.}


\begin{document}

\begin{abstract}
  We present a procedure to construct $(n+1)$-Hom-Nambu-Lie algebras
  from $n$-Hom-Nambu-Lie algebras equipped with a generalized trace
  function. It turns out that the implications of the compatibility
  conditions, that are necessary for this construction, can be
  understood in terms of the kernel of the trace function and the
  range of the twisting maps. Furthermore, we investigate the
  possibility of defining $(n+k)$-Lie algebras from $n$-Lie algebras
  and a $k$-form satisfying certain conditions.
\end{abstract}

\maketitle

\section{Introduction}\label{sec:intro}

\noindent Lie algebras and Poisson algebras have played an extremely
important role in mathematics and physics for a long time. Their
generalizations, known as $n$-Lie algebras and ``Nambu algebras''
\cite{f:nliealgebras,n:generalizedmech,Takhtajan} also arise naturally
in physics and have, for instance, been studied in the context of
``M-branes'' \cite{BL2007,HoppeJMalgNambumech}. Moreover, it has
recently been shown that the differential geometry of $n$-dimensional
Riemannian submanifolds can be described in terms of an $n$-ary Nambu
algebra structure on the space of smooth functions on the manifold
\cite{ahh:multilinear}.

A long-standing problem related to Nambu algebras is their
quantization. For Poisson algebras, the problem of finding an operator
algebra where the commutator Lie algebra corresponds to the Poisson
algebra is a well-studied problem, e.g. in the context of matrix
regularizations
\cite{abhhs:fuzzy,abhhs:noncommutative,a:phdthesis,a:repcalg,as:affinecrossed}.
For higher order algebras much less is known and the corresponding
problem seems to be difficult to study. A Nambu-Lie algebra is defined
in general by an $n$-ary multilinear multiplication which is skew-symmetric and satisfies an identity extending the Jacobi identity for
the Lie algebras. For $n=3$ this identity is
\begin{equation*}\label{TernaryNambuIdentityOrdinary}
  [ x_{1},x_{ 2}, [x_{3},x_{4},x_{5}]] =
  [ [x_1,x_2,x_3 ],x_{4},x_{5}]
  + [ x_3,[ x_1,x_2,x_{4}] , x_{5}]+
  [x_3,x_{4},[x_1,x_2,x_{5}]].
\end{equation*}
In Nambu-Lie algebras, the additional freedom in comparison with Lie
algebras is mainly limited to extra arguments in the multilinear
multiplication. The identities of Nambu-Lie algebras are also
closely resembling the identities for Lie algebras. As a result,
there is a close similarity between Lie algebras and Nambu-Lie
algebras in their appearances in connection to other algebraic and
analytic structures and in the extent of their applicability.  Thus
it is not surprising that it becomes unclear how to associate in
meaningful ways ordinary Nambu-Lie algebras with the important in
physics generalizations and quantum deformations of Lie algebras
when typically the ordinary skew-symmetry and Jacobi identities of
Lie algebras are violated. However, if the class of Nambu-Lie
algebras is extended with enough extra structure beyond just adding
more arguments in multilinear multiplication, the natural ways of
association of such multilinear algebraic structures with
generalizations and quantum deformations of Lie algebras may become
feasible. Hom-Nambu-Lie algebras are defined by a similar but more
general identity than that of Nambu-Lie algebras involving some
additional linear maps. These linear maps twisting or deforming
the main identities introduce substantial new freedom in the
structure allowing to consider Hom-Nambu-Lie algebras as
deformations of Nambu-Lie algebras ($n$-Lie algebras). The extra
freedom built into the structure of Hom-Nambu-Lie algebras may
provide a path to quantization beyond what is possible for ordinary
Nambu-Lie algebras.  All this gives also important motivation for
investigation of mathematical concepts and structures such as
Leibniz $n$-ary algebras \cite{CassasLodayPirashvili,f:nliealgebras}
and their modifications and extensions, as well as Hom-algebra
extensions of Poisson algebras \cite{HomDeform}. For discussion of physical
applications of these and related algebraic structures to models for
elementary particles, and unification problems for interactions see
\cite{Kerner7,Kerner,Kerner2,Kerner4,Kerner6}.

The general Hom-algebra structures arose first in connection to
quasi-deformation and discretizations of Lie algebras of vector
fields \cite{HLS, LS2}. These quasi-deformations lead to quasi-Lie algebras, quasi-Hom-Lie algebras and Hom-Lie algebras, 
which are generalized Lie algebra structures with twisted skew-symmetry and
Jacobi conditions.  
The first motivating examples in physics and mathematics literature 
are $q$-deformations of the Witt and Virasoro algebras constructed in the investigations of vertex models in   
conformal field theory \cite{AizawaSaito,ChaiElinPop,ChaiKuLukPopPresn,ChaiKuLuk,ChaiPopPres,CurtrZachos1,
DaskaloyannisGendefVir,Hu,LiuKQuantumCentExt,LiuKQCharQuantWittAlg,LiuKQPhDthesis}.  
Motivated by these and new examples arising as applications of the general quasi-deformation
construction of \cite{HLS,LS1,LS2} on the one hand, and the desire to
be able to treat within the same framework such well-known
generalizations of Lie algebras as the color and super Lie algebras on
the other hand, quasi-Lie algebras and subclasses of quasi-Hom-Lie
algebras and Hom-Lie algebras were introduced in
\cite{HLS,LS1,LS2,LS3,MS}. 
In Hom-Lie algebras,
skew-symmetry is untwisted, whereas the Jacobi identity is twisted by
a single linear map and contains three terms as for Lie algebras,
reducing to the Jacobi identity for ordinary Lie algebras when the linear twisting map is the
identity map. 

In this paper, we will be concerned with
$n$-Hom-Nambu-Lie algebras, a class of $n$-ary algebras generalizing $n$-ary algebras of Lie type
including $n$-ary Nambu algebras, $n$-ary Nambu-Lie algebras and
$n$-ary Lie algebras \cite{ams:gennambu,ams:ternary}.
In \cite{ams:ternary}, a method was demonstrated of how to construct
ternary multiplications from the binary multiplication of a Hom-Lie algebra, a linear twisting map, and a trace function satisfying certain compatibility conditions; and it was shown that this method can be used to construct
ternary Hom-Nambu-Lie algebras from Hom-Lie algebras.
In this article we extend the results and the binary-to-ternary construction of \cite{ams:ternary} 
to the general case of $n$-ary algebras.  
This paper is organized as follows.  In Section \ref{sec:prel} we
review basic concepts of Hom-Lie, and $n$-Hom-Nambu-Lie algebras. In
Section \ref{sec:construction} we provide a construction procedure of
$(n+1)$-Hom-Nambu-Lie algebras starting from an $n$-Hom-Nambu-Lie algebra
and a trace function satisfying certain compatibility conditions
involving the twisting maps. To this end, we use the ternary bracket
introduced in \cite{amdy:quantnambu}.  In Section
\ref{sec:compatible}, we investigate how restrictive the
compatibility conditions are. The mutual position of kernels of
twisting maps and the trace play an important role in this context.
Finally, in Section \ref{sec:higherorder}, we investigate the
possibility to define $(n+k)$-Lie algebras starting from an $n$-Lie
algebra and a $k$-form satisfying certain conditions.

\section{Preliminaries}\label{sec:prel}

In \cite{ams:gennambu}, generalizations of $n$-ary algebras of Lie
type and associative type by twisting the identities using linear maps
have been introduced. These generalizations include $n$-ary
Hom-algebra structures generalizing the $n$-ary algebras of Lie type
including $n$-ary Nambu algebras, $n$-ary Nambu-Lie algebras and
$n$-ary Lie algebras, and $n$-ary algebras of associative type
including $n$-ary totally associative and $n$-ary partially
associative algebras.

\begin{definition} \label{def:HomLie} A \emph{Hom-Lie algebra}
  $(V,[\cdot, \cdot], \alpha)$ is a vector space $V$ together with a
  skew-symmetric bilinear map $[\cdot, \cdot]: V\times V \rightarrow
  V$ and a linear map $\alpha: V \rightarrow V$ satisfying
  \begin{align*}
    [\alpha(x),[y,z]] = [[x,y],\alpha(z)] + [\alpha(y),[x,z]]
  \end{align*}
  for all $x,y,z\in V$.
\end{definition}

\begin{definition}\label{def:nHomNambuLie}
  A $n$-Hom-Nambu-Lie algebra
  $(V,[\cdot,\ldots,\cdot],\alpha_1,\ldots,\alpha_{n-1})$ is a vector
  space $V$ together with a skew-symmetric multilinear map
  $[\cdot,\ldots,\cdot]:V^n\to V$ and linear maps
  $\alpha_1,\ldots,\alpha_{n-1}:V\to V$ such that
  \begin{align*}
    &[\alpha_1(x_1),\ldots,\alpha_{n-1}(x_{n-1}),[y_1,\ldots,y_n]]\\
    &= \sum_{k=1}^n [\alpha_1(y_1),\ldots,\alpha_{k-1}(y_{k-1}),
    [x_1,\ldots,x_{n-1},y_k],\alpha_k(y_{k+1}),\ldots,\alpha_{n-1}(y_n)]
  \end{align*}
  for all $x_1,\ldots,x_{n-1},y_1,\ldots,y_n\in V$.  The linear maps
  $\alpha_1,\ldots,\alpha_{n-1}$ are called the \emph{twisting maps}
  of the Hom-Nambu-Lie algebra. A $n$-Lie algebra is an $n$-Hom-Lie
  algebra with $\alpha_1=\alpha_2=\cdots=\alpha_{n-1}=\id_V$.
\end{definition}

\section{Construction of Hom-Nambu-Lie algebras}\label{sec:construction}

\noindent In \cite{ams:ternary}, the authors introduced a procedure to
induce $3$-Hom-Nambu-Lie algebras from Hom-Lie algebras. In the
following, we shall extend this procedure to induce a
$(n+1)$-Hom-Nambu-Lie algebra from an $n$-Hom-Nambu-Lie algebra.  Let
us start by defining the skew-symmetric map that will be used to
induce the higher order algebra.  In the following, $\K$ denotes a field
of characteristic 0, and $V$ a vector space over $\K$.

\begin{definition}\label{def:phitau}
  Let $\phi:V^n\to V$ be an $n$-linear map and let $\tau$ be a map from
  $V$ to $\K$. Define $\phit:V^{n+1}\to V$ by
  \begin{align}
    \phit(x_1,\ldots,x_{n+1}) = \sum_{k=1}^{n+1}(-1)^k\tau(x_k)\phi(x_1,\ldots,\xh_k,\ldots,x_{n+1}),
  \end{align}
where the hat over $\xh_k$ on the right hand side means that $x_{k}$ is excluded, that is 
 $\phi$ is calculated on $(x_1,\ldots, x_{k-1}, x_{k+1}, \ldots, x_{n+1})$. 
\end{definition}
We will not be concerned with just any linear map $\tau$,
but rather maps that have a generalized trace property. Namely

\begin{definition}\label{def:phitrace}
  For $\phi:V^n\to V$ we call a linear map $\tau:V\to\K$ a
  \emph{$\phi$-trace} if $\tau\paraa{\phi(x_1,\ldots,x_n)}=0$ for all
  $x_1,\ldots,x_n\in V$.
\end{definition}

\begin{lemma}\label{lemma:phitlinantisym}
  Let $\phi:V^n\to V$ be a totally skew-symmetric $n$-linear map and
  $\tau$ a linear map $V\to\K$. Then $\phit$ is a $(n+1)$-linear
  totally skew-symmetric map. Furthermore, if $\tau$ is a $\phi$-trace
  then $\tau$ is a $\phit$-trace.
\end{lemma}

\begin{proof}
  The $(n+1)$-linearity property of $\phit$ follows from $n$-linearity of $\phi$ and linearity of $\tau$
  as it is a linear combination of $(n+1)$-linear maps 
  $$\tau(x_k)\phi(x_1,\ldots,\xh_k,\ldots,x_{n+1}), \quad 1\leq k \leq n+1.$$
  To prove total skew-symmetry one simply notes that
  \begin{align*}
    \phit(x_1,\ldots,x_{n+1}) = -\frac{1}{n!}\sum_{\sigma\in S_{n+1}}\sgn(\sigma)
    \tau(x_{\sigma(1)})\phi(x_{\sigma(2)},\ldots,x_{\sigma(n+1)}),
  \end{align*}
  which is clearly skew-symmetric. Since each term in $\phit$ is
  proportional to $\phi$ and since $\tau$ is linear and a
  $\phi$-trace, $\tau$ will also be a $\phit$-trace.
\end{proof}

\noindent The extension of Theorem 3.3 in \cite{ams:ternary} can now
be formulated as follows:

\begin{theorem}\label{thm:inducedhomnambulie}
  Let $(V,\phi,\alpha_1,\ldots,\alpha_{n-1})$ be a
  $n$-Hom-Nambu-Lie algebra, $\tau$ a $\phi$-trace and
  $\alpha_n:V\to V$ a linear map. If it holds that
  \begin{align}
    &\tau\paraa{\alpha_i(x)}\tau(y)=\tau(x)\tau\paraa{\alpha_i(y)}\label{eq:tautaualpha}\\
    &\tau\paraa{\alpha_i(x)}\alpha_j(y)=\alpha_i(x)\tau\paraa{\alpha_j(y)}\label{eq:taualphaalpha}
  \end{align}
  for all $i,j\in\{1,\ldots,n\}$ and all $x,y\in V$, then
  $(V,\phit,\alpha_1,\ldots,\alpha_n)$ is a $(n+1)$-Hom-Nambu-Lie
  algebra. We shall say that $(V,\phit,\alpha_1,\ldots,\alpha_n)$ is
  \emph{induced from $(V,\phi,\alpha_1,\ldots,\alpha_{n-1})$}.
\end{theorem}

\begin{proof}
  Since $\phit$ is skew-symmetric and multilinear by Lemma
  \ref{lemma:phitlinantisym}, one only has to check that the
  Hom-Nambu-Jacobi identity is fulfilled. This identity is written as
  \begin{align*}
    \sum_{s=1}^{n+1}\phit\paraa{\alpha_1(u_1),&\ldots,\alpha_{s-1}(u_{s-1}),\phit(x_1,\ldots,x_n,u_s),\alpha_s(u_{s+1}),\ldots,\alpha_n(u_{n+1})}\\
    &-\phit\paraa{\alpha_1(x_1),\ldots,\alpha_n(x_n),\phit(u_1,\ldots,u_{n+1})}=0.
  \end{align*}
  Let us write the left-hand-side of this equation as $A-B$ where
  \begin{align*}
    A &= \sum_{s=1}^{n+1}\phit\paraa{\alpha_1(u_1),\ldots,\alpha_{s-1}(u_{s-1}),\phit(x_1,\ldots,x_n,u_s),\alpha_s(u_{s+1}),\ldots,\alpha_n(u_{n+1})}\\
    B &= \phit\paraa{\alpha_1(x_1),\ldots,\alpha_n(x_n),\phit(u_1,\ldots,u_{n+1})}.
  \end{align*}
  Furthermore, we expand $B$ into terms $B_{kl}$ such that
  \begin{align*}
    &B_{kl} = (-1)^{k+l} \tau\paraa{\alpha_k(x_k)}\tau(u_l) \times \\
    & \qquad\qquad\qquad\qquad  \times \phi\paraa{\alpha_1(x_1),\ldots,\widehat{\alpha_{k}(x_k)},\ldots,\alpha_n(x_n),\phi(u_1,\ldots,\hat{u}_l,\ldots,u_{n+1})}\\
    &B = \sum_{k=1}^{n}\sum_{l=1}^{n+1} B_{kl},
  \end{align*}
  taking into account that $\tau$ is a $\phi$-trace form. We expand $A$ as
  \begin{align*}
    A = &(-1)^{n+1}\sum_{s=1}^{n+1}\tau(u_s)\phit\paraa{\alpha_1(u_1),\ldots,\phi(x_1,\ldots,x_n),\ldots,\alpha_n(u_{n+1})}\\
    &+\sum_{s=1}^{n+1}\sum_{k=1}^n(-1)^k\tau(x_k)\phit\paraa{\alpha_1(u_1),\ldots,\phi(x_1,\ldots,\hat{x}_{k},\ldots,x_n,u_s),\ldots,\alpha_n(u_{n+1})}\\
    &\equiv (-1)^{n+1}A^{(1)}+A^{(2)}.
  \end{align*}
  Let us now show that $A^{(1)}=0$. For every choice of integers
  $k<l\in\{1,\ldots,n+1\}$, $A^{(1)}$ contains two terms where one
  $\tau$ involves $u_k$ and the other $\tau$ involves $u_l$. Namely,
  \begin{align*}
    A^{(1)} = &\sum_{k<l=1}^{n+1}(-1)^l\tau(u_k)\tau\paraa{\alpha_{l-1}(u_l)}\times\\
    &\quad\times\phi\paraa{\alpha_1(u_1),\ldots,\alpha_{k-1}(u_{k-1}),\phi(x_1,\ldots,x_n),\ldots,\widehat{\alpha_{l-1}(u_l)},\ldots,\alpha_n(u_{n+1})}\\
    &\quad+\sum_{k<l=1}^{n+1}(-1)^k\tau(u_l)\tau\paraa{\alpha_k(u_k)}\times\\
    &\quad\quad\times\phi\paraa{\alpha_1(u_1),\ldots,\widehat{\alpha_{k}(u_k)},\ldots,\phi(x_1,\ldots,x_n),\alpha_{l}(u_{l+1}),\ldots,\alpha_{n}(u_{n+1})}.
  \end{align*}
  By using relations (\ref{eq:tautaualpha}) and (\ref{eq:taualphaalpha}) one can write these two terms together as
  \begin{align*}
    A^{(1)}=&\sum_{k<l=1}^{n+1}(-1)^l\bracketb{\tau(u_k)\tau\paraa{\alpha_{l-1}(u_l)}-\tau(u_l)\tau\paraa{\alpha_{l-1}(u_k)}}\times\\
    &\quad\times
    \phi\paraa{\alpha_1(u_1),\ldots,\alpha_{k-1}(u_{k-1}),\phi(x_1,\ldots,x_n),\ldots,\widehat{\alpha_{l-1}(u_l)},\ldots,\alpha_n(u_{n+1})},
  \end{align*}
  and it follows from (\ref{eq:tautaualpha}) that each
  term in this sum is zero. Let us now consider the expansion of $A^{(2)}$ via
  \begin{align*}
    &A^{(2)}_{kl}=\sum_{s=1}^{l-1}(-1)^{k+l}\tau(x_k)\tau\paraa{\alpha_{l-1}(u_l)}\times\\
    &\qquad\qquad\times\phi\paraa{\alpha_1(u_1),\ldots,\phi(x_1,\ldots,\hat{x}_k,\ldots,x_n,u_s),\ldots,\widehat{\alpha_{l-1}(u_l)},\ldots,\alpha_n(u_{n+1})}\\
    &\qquad\quad+\sum_{s=l+1}^{n+1}(-1)^{k+l}\tau(x_k)\tau\paraa{\alpha_{l}(u_l)}\times\\
    &\qquad\qquad\qquad\times\phi\paraa{\alpha_1(u_1),\ldots,\widehat{\alpha_l(u_l)},\ldots,\phi(x_1,\ldots,\hat{x}_k,\ldots,x_n,u_s),\ldots,\alpha_n(u_{n+1})}.
  \end{align*}
  One notes that relations (\ref{eq:tautaualpha}) and
  (\ref{eq:taualphaalpha}) allows one to swap any $\alpha_i$ and
  $\alpha_j$ in the expression above. Therefore, $B_{kl}-A^{(2)}_{kl}$
  will be proportional to the Hom-Nambu-Jacobi identity for $\phi$
  (acting on the elements
  $x_1,\ldots,\hat{x}_k,\ldots,x_n,u_1,\ldots,\hat{u}_l,\ldots,u_{n+1}$)
  since one can make sure that $\alpha_n$, which is not one of the
  twisting maps of the original Hom-Nambu-Lie algebra, appears outside
  any bracket. Hence, $A^{(2)}-B=0$ and the Hom-Nambu-Jacobi identity
  is satisfied for $\phit$.
\end{proof}

\noindent Since $\tau$ is also a $\phit$-trace, one can repeat the
procedure in Theorem \ref{thm:inducedhomnambulie} to induce a
$(n+2)$-Hom-Nambu-Lie algebra from an $n$-Hom-Nambu-Lie
algebra. However, the result is an abelian algebra.

\begin{proposition}\label{prop:doubleAbelian}
  Let $\A=(V,\phi,\alpha_1,\ldots,\alpha_{n-1})$ be a
  $n$-Hom-Nambu-Lie algebra. Let $\A'$ be any $(n+1)$-Hom-Nambu-Lie
  algebra induced from $\A$ via the $\phi$-trace $\tau$. If $\A''$ is
  a $(n+2)$-Hom-Nambu-Lie algebra induced from $\A'$ using the same 
  $\tau$ again, then $\A''$ is abelian.
\end{proposition}

\begin{proof}
  By the definition of $\phit$, the bracket on the algebra $\A''$ can be
  written as
  \begin{align*}
    \paraa{\phit}_\tau = \sum_{k=1}^{n+2}(-1)^k\tau(x_k)\phit(x_1,\ldots,\xh_k,\ldots,x_{n+2}).
  \end{align*}
  Expanding the bracket $\phit$, there will be, for every choice
  of integers $k<l$, two terms which are proportional
  to $\tau(x_k)\tau(x_l)$. Their sum becomes
  \begin{align*}
    \tau(x_k)\tau(x_l)\phi(x_1,\ldots,\xh_k,\ldots,\xh_l,\ldots,x_{n+2})
    \parab{(-1)^{k+l}+(-1)^{k+l-1}} = 0.
  \end{align*}
  Hence, $(\phit)_\tau(x_1,\ldots,x_{n+2})=0$ for all
  $x_1,\ldots,x_{n+2}\in V$.
\end{proof}

\noindent Note that one might choose a different $\phit$-trace when repeating
the construction, and this can lead to non-abelian algebras as in the
following example.
\begin{example}
  Let us use an example from \cite{ams:ternary}, where one starts with
  a Hom-Lie algebra defined on a vector space
  $V=\spanvec(x_1,x_2,x_3,x_4)$ via
  \begin{align*}
    [x_i,x_j] = a_{ij}x_3+b_{ij}x_4
  \end{align*}
  and $\alpha_1(x_i) = x_3$ for $i=1,\ldots,4$. This will be a Hom-Lie
  algebra provided $a_{ij},b_{ij}$ satisfy certain conditions; let us
  for definiteness choose
  \begin{align*}
    (b_{ij})=
    \begin{pmatrix}
      0    & b & b & b+c \\
      -b   & 0 & 0 & c \\
      -b   & 0 & 0 & c \\
      -b-c & -c & -c & 0
    \end{pmatrix}
  \end{align*}
  together with $a_{i j}=1$ for all $i<j$. To construct a
  $3$-Hom-Nambu-Lie algebra we set $\tau(x_3)=\tau(x_4)=0$,
  $\tau(x_1)=\tau(x_2)=1$ and $\alpha_2(x_i)=x_4$ for
  $i=1,\ldots,4$. One can easily check that $\tau$ is a $\phi$-trace
  and that the compatibility conditions are fulfilled. The induced
  algebra will then have the following brackets:
  \begin{align*}
    &[x_1,x_2,x_3] = -bx_4\\
    &[x_1,x_2,x_4] = -cx_4\\
    &[x_1,x_3,x_4] = x_3+c x_4\\
    &[x_2,x_3,x_4] = x_3+c x_4.
  \end{align*}
  Now, we want to continue this process and find another trace $\rho$
  together with a linear map $\alpha_3$ such that the resulting
  $4$-Hom-Nambu-Lie algebra is non-abelian. By choosing
  $\rho(x_3)=\rho(x_4)=0$, $\rho(x_1)=\delta_1$, $\rho(x_2)=\delta_2$
  and $\alpha_3=\alpha_1$, one sees that $\rho$ is a trace and that
  the compatibility conditions are fulfilled. The induced
  $4$-Hom-Nambu-Lie algebra has only one independent bracket,
  namely
  \begin{align*}
    [x_1,x_2,x_3,x_4] = (\delta_2-\delta_1)(x_3+cx_4),
  \end{align*}
  which is non-zero for $\delta_1\neq\delta_2$.

\end{example}

\section{The compatibility conditions}\label{sec:compatible}

\noindent Given an $n$-Hom-Nambu-Lie algebra we ask the question: Can
we find a trace and a linear map such that a $(n+1)$-Hom-Nambu-Lie
algebra can be induced? In the following we shall study the
implications of the assumptions in Theorem
\ref{thm:inducedhomnambulie}; it turns out that the relation between
the kernel of $\tau$ and the range of $\alpha_i$ is important.

\begin{definition}\label{def:compatible}
  Let $V$ be a vector space, $\alpha_1,\ldots,\alpha_n$ linear maps
  $V\to V$ and $\tau$ a linear map $V\to\K$. The tuple
  $\alphatautuple$ is \emph{compatible on $V$} if
  \begin{align}
    &\tau\paraa{\alpha_i(x)}\tau(y)=\tau(x)\tau\paraa{\alpha_i(y)}\\
    &\tau\paraa{\alpha_i(x)}\alpha_j(y)=\alpha_i(x)\tau\paraa{\alpha_j(y)}
  \end{align}
  for all $x,y\in V$ and $i,j\in\{1,\ldots,n\}$. A compatible tuple is
  \emph{nondegenerate} if $\ker(\tau)\neq V$ and
  $\ker(\tau)\neq\{0\}$.
\end{definition}

\noindent We introduce $K=\ker(\tau)$ and $U=V\backslash K$. Note that
for a nondegenerate compatible tuple $U$ is always non-empty and $K$
contains at least one non-zero element.

\begin{lemma}\label{lemma:alphaKernel}
  If $\alphatautuple$ is a nondegenerate compatible tuple on $V$ then
  $\alpha_i(K)\subseteq K$ for $i=1,\ldots,n$.
\end{lemma}

\begin{proof}
  Let $x$ be an arbitrary element of $K$. Since the tuple is assumed
  to nondegenerate, there exists a non-zero element $y\in
  U$. Relation (\ref{eq:tautaualpha}) applied to $x$ and $y$ gives
  \begin{align*}
    \tau(y)\tau\paraa{\alpha_i(x)} = 0,
  \end{align*}
  and since $\tau(y)\neq 0$ this implies that
  $\tau\paraa{\alpha_i(x)}=0$.
\end{proof}

\begin{lemma}\label{lemma:existsu}
  Let $\alphatautuple$ be a nondegenerate compatible tuple on $V$ and
  assume that there exists an element $u\in U$ such that $\alpha(u)\in
  K$. Then $\alpha(V)\subseteq K$.
\end{lemma}

\begin{proof}
  Let $x$ be an arbitrary element of $V$. Relation
  (\ref{eq:tautaualpha}) applied to $x$ and $u$ gives
  \begin{align*}
    \tau(u)\tau\paraa{\alpha_i(x)} = 0,
  \end{align*}
  and since $\tau(u)\neq 0$ it follows that $\tau\paraa{\alpha_i(x)}=0$.
\end{proof}

\noindent Hence, for a nondegenerate compatible tuple it either holds
that $\alpha_i(V)\subseteq K$ or $\alpha_i(U)\subseteq U$.

\begin{proposition}\label{prop:imageUU}
  Let $\alphatautuple$ be a nondegenerate compatible tuple on $V$ and
  assume that there exist $i,j\in\{1,\ldots,n\}$ such that
  $\alpha_i(U)\subseteq U$ and $\alpha_j(U)\subseteq U$. Then there
  exists $\lambda_{ij}\in\K\backslash\{0\}$ such that
  $\alpha_i=\lambda_{ij}\alpha_j$, where
  $\lambda_{ij}=\tau\paraa{\alpha_i(u)}/\tau\paraa{\alpha_j(u)}$ for
  any $u\in U$.
\end{proposition}

\begin{proof}
  With $u\in U$ and $x\in V$ equation (\ref{eq:taualphaalpha}) becomes
  \begin{align*}
    \tau\paraa{\alpha_i(u)}\alpha_j(x) = \tau\paraa{\alpha_j(u)}\alpha_i(x).
  \end{align*}
  By assumption, $\tau\paraa{\alpha_i(u)}\neq 0$ and
  $\tau\paraa{\alpha_j(u)}\neq 0$, which implies that
  \begin{align*}
    \alpha_i(x) =
    \frac{\tau\paraa{\alpha_i(u)}}{\tau\paraa{\alpha_j(u)}}\alpha_j(x),
  \end{align*}
  which proves the statement.
\end{proof}

\begin{proposition}\label{prop:imageUK}
  Let $\alphatautuple$ be a nondegenerate compatible tuple on $V$ and
  assume there exist $i,j\in\{1,\ldots,n\}$ such that
  $\alpha_i(U)\subseteq U$ and $\alpha_j(U)\subseteq K$. Then
  $\alpha_j(x)=0$ for all $x\in V$.
\end{proposition}

\begin{proof}
  Assume that $\alpha_i(U)\subseteq U$ and $\alpha_j(U)\subseteq K$
  and let $u\in U$ and $x\in V$. Equation (\ref{eq:taualphaalpha})
  gives
  \begin{align*}
    \tau\paraa{\alpha_i(u)}\alpha_j(x) = 0,
  \end{align*}
  which implies that $\alpha_j(x)=0$ since $\alpha_i(u)\in U$.
\end{proof}

\noindent The above results tell us that given an $n$-Hom-Nambu-Lie
algebra with twisting maps $\alpha_1,\ldots,\alpha_{n-1}$, and a
$\phi$-trace $\tau$, there is not much choice when choosing
$\alpha_n$. Namely, if there is a twisting map $\alpha_i$ such that
$\alpha_i(U)\subseteq U$ then either $\alpha_n=0$ or $\alpha_n$ is
proportional to $\alpha_i$. In the case when $\alpha_i(U)\subseteq K$
for $i=1,\ldots,n-1$ one has slightly more freedom of choosing
$\alpha_n$, but note that unless $\alpha_n(U)\subseteq K$, Proposition
\ref{prop:imageUK} gives $\alpha_i=0$ for $i=1,\ldots,n-1$. When
$\alpha_i(U)\subseteq K$ for $i=1,\ldots,n$ the compatibility
conditions (\ref{eq:tautaualpha}) and (\ref{eq:taualphaalpha}) are
automatically satisfied.

Hence, there are two potentially interesting cases which give rise to
non-zero twisting maps:
\begin{align}
  &\alpha_i(U)\subseteq U\text{ for }i=1,\ldots,n\tag{C1}\label{eq:C1}\\
  &\alpha_i(U)\subseteq K\text{ for }i=1,\ldots,n.\tag{C2}\label{eq:C2}
\end{align}
In case (\ref{eq:C1}) all the twisting maps will be proportional.

Having considered the case of nondegenerate compatible tuples, let us
show that degenerate ones lead to abelian algebras.

\begin{proposition}\label{prop:degenerateAbelian}
  Let $\A=(V,\phi_\tau,\alpha_1,\ldots,\alpha_n)$ be a Hom-Nambu-Lie
  algebra induced by $(V,\phi,\alpha_1,\ldots,\alpha_{n-1})$ and a
  $\phi$-trace $\tau$. If $\alphatautuple$ is a degenerate compatible
  tuple then $\A$ is abelian.
\end{proposition}

\begin{proof}
  By the definition of $\phi_\tau$ it is clear that if $\ker\tau=V$
  then $\phi_\tau(x_1,\ldots,x_{n+1})=0$ for all
  $x_1,\ldots,x_{n+1}\in V$. Now, assume that $\ker\tau=\{0\}$. Since
  $\tau$ is a $\phi$-trace one has that
  $\tau\paraa{\phi(x_1,\ldots,x_n)}=0$ for all $x_1,\ldots,x_n\in
  V$. Hence, $\phi(x_1,\ldots,x_n)$ is in $\ker\tau$ which implies
  that $\phi(x_1,\ldots,x_{n})=0$. From this it immediately follows
  that $\phi_\tau(x_1,\ldots,x_{n+1})=0$ for all $x_1,\ldots,x_{n+1}\in
  V$.
\end{proof}

\section{Twisting of $n$-ary Hom-Nambu-Lie algebras}\label{sec:twisting}

\noindent In \cite{y:nhomliegoingdown} the general property that an $n$-Lie algebra induces a $(n-k)$-Lie algebra by fixing $k$ elements
in the bracket, was extended to Hom-Nambu-Lie algebras.

\begin{definition}\label{dec:interior}
  Let $\phi:V^n\to V$ be a linear map and let $a_1,\ldots,a_{k}$ (with
  $k<n$) be elements of $V$. By $\pi_{a_1\cdots a_k}\phi$ we denote the
  map $V^{n-k}\to V$ defined by
  \begin{align}
    \paraa{\pi_{a_1\cdots a_k}\phi}(x_1,\ldots,x_{n-k}) = \phi(x_1,\ldots,x_{n-k},a_1,\ldots,a_k).
  \end{align}
\end{definition}

\noindent The result in \cite{y:nhomliegoingdown} can now be stated as

\begin{proposition}\label{prop:yaugoingdown}
  Let $(V,\phi,\alpha_1,\ldots,\alpha_{n-1})$ be an $n$-Hom-Nambu-Lie
  algebra and let $a_1,\ldots,a_k\in V$ (with $k<n$) be elements such
  that $\alpha_{n-k-1+i}(a_{i})=a_{i}$ for $i=1,\ldots,k$. Then
  $(V,\pi_{a_1\cdots a_k}\phi,\alpha_{1},\ldots,\alpha_{n-k-1})$ is a
  $(n-k)$-Hom-Nambu-Lie algebra.
\end{proposition}

\noindent Hence, given an $n$-Hom-Nambu-Lie algebra one can create a ``twisted''
$n$-Hom-Nambu-Lie algebra by applying Proposition
\ref{prop:yaugoingdown} and Theorem \ref{thm:inducedhomnambulie}.

\begin{proposition}\label{prop:updown}
  Let $(V,\phi,\alpha_1,\ldots,\alpha_{n-1})$ be an $n$-Hom-Nambu-Lie
  algebra, $\tau$ a $\phi$-trace and $\alpha_n:V\to V$ a linear map
  such that equations (\ref{eq:tautaualpha}) and
  (\ref{eq:taualphaalpha}) are fulfilled. If there exists an element
  $a\in V$ such that $\alpha_n(a)=a$ then
  $(V,\pi_a\phit,\alpha_1,\ldots,\alpha_{n-1})$ is an $n$-Hom-Nambu-Lie
  algebra with
  \begin{align*}
    \paraa{\pi_a\phit}(x_1,\ldots,x_n) =
    \sum_{k=1}^{n}&(-1)^k\tau(x_k)\phi(x_1,\ldots,\xh_k,\ldots,x_{n},a)\\
    &+(-1)^{n+1}\tau(a)\phi(x_1,\ldots,x_n).
  \end{align*}
\end{proposition}

\noindent One can also go the other way around: first applying $\pi$
and then inducing a higher order algebra.

\begin{proposition}\label{prop:downup}
  Let $(V,\phi,\alpha_1,\ldots,\alpha_{n-1})$ be an $n$-Hom-Nambu-Lie
  algebra and $\tau$ a $\pi_a\phi$-trace such that equations
  (\ref{eq:tautaualpha}) and (\ref{eq:taualphaalpha}) are
  fulfilled. If there exists an element $a\in V$ such that
  $\alpha_{n-1}(a)=a$ then it holds that
  $(V,(\pi_a\phi)_\tau,\alpha_1,\ldots,\alpha_{n-1})$ is a
  $n$-Hom-Nambu-Lie algebra with
  \begin{align*}
    \paraa{\pi_a\phi}_\tau(x_1,\ldots,x_n) =
    \sum_{k=1}^{n}(-1)^k\tau(x_k)\phi(x_1,\ldots,\xh_k,\ldots,x_{n},a).
  \end{align*}
\end{proposition}

\noindent One notes that the two types of twistings are in general not
equivalent. In fact, assuming the two procedures in Propositions
\ref{prop:updown} and \ref{prop:downup} are well defined for some
element $a\in V$ and denoting $\phit\equiv i_\tau\phi$, one can write
\begin{align}
  \paraa{[i_\tau,\pi_a]\phi}(x_1,\ldots,x_n) = (-1)^n\tau(a)\phi(x_1,\ldots,x_n).
\end{align}
Thus, unless $a\in\ker\tau$, one recovers the ``untwisted''
$n$-Hom-Nambu-Lie algebra as the commutator of the maps $i_\tau$ and
$\pi_a$. If $a\in\ker\tau$ then the two procedures yield the same
result.

Recalling the possible cases for the relation between
$\alpha_1,\ldots,\alpha_n$ and the kernel of $\tau$, one notes that in
case (\ref{eq:C2}) any fixed point of $\alpha_i$ is necessarily in the
kernel of $\tau$, which implies that $[i_\tau,\pi_a]=0$. In case
(\ref{eq:C1}) there might be fixed points in $U$.

\section{Higher order constructions}\label{sec:higherorder}

\noindent A natural extension of the current framework would be to
construct a $(n+p)$-Lie algebra from an $n$-Lie algebra and a $p$-form
by using the wedge product. Let us illustrate how this can be done and
investigate the connection to a closely related kind of algebras,
satisfying the so called \emph{generalized Jacobi identity}.

\begin{definition}\label{def:tauwedgephi}
  Let $(V,\phi,\alpha_1,\ldots,\alpha_{n-1})$ be an $n$-ary
  Hom-Nambu-Lie algebra and let $\tau\in\wedge^p V$ be a
  $p$-form. Define $\tau\wedge\phi:V^{n+p}\to V$ by
  \begin{align*}
    \paraa{\tau\wedge\phi}(x_1,\ldots,x_{n+p}) &=
    \frac{1}{n!p!}\sum_{\sigma\in S_{n+p}}\sgn(\sigma)\tau(x_{\sigma(1)},\ldots,x_{\sigma(p)})\phi(x_{\sigma(p+1)},\ldots,x_{\sigma(n+p)})\\
    &\equiv \phit(x_1,\ldots,x_{n+p})
  \end{align*}
  for all $x_1,\ldots,x_{n+p}\in V$.
\end{definition}

\noindent There is a natural extension of the concept of $\phi$-traces
to $p$-forms.

\begin{definition}\label{def:phicompatiblepform}
  For $\phi:V^n\to V$ we call $\tau\in\wedge^p V$ a
  \emph{$\phi$-compatible $p$-form} if
  $\tau(\phi(x_1,\ldots,x_n),y_1,\ldots,y_{p-1})=0$ for all
  $x_1,\ldots,x_n,y_1,\ldots,y_{p-1}\in V$.
\end{definition}

\noindent The fundamental identity of $n$-Lie algebras is not a
complete symmetrization of an iterated bracket. The generalized Jacobi
identity is a more symmetric extension of the standard Jacobi
identity.

\begin{definition}\label{def:gji}
  A vector valued form $\phi\in\wedge^n(V,V)$ is said to satisfy the
  \emph{generalized Jacobi identity} if
  \begin{align}
    \sum_{\sigma\in S_{2n-1}}\sgn(\sigma)\phi\paraa{\phi(x_{\sigma(1)},\ldots,x_{\sigma(n)}),x_{\sigma(n+1)},\ldots,x_{\sigma(2n-1)}}=0
  \end{align}
  for all $x_1,\ldots,x_{2n-1}\in V$.
\end{definition}

\noindent In the following we study the question when
$\tau\wedge\phi$ satisfies the generalized Jacobi identity, or the
fundamental identity of $n$-Lie algebras.

For $\phi\in\wedge^{k}(V,V)$ and $\psi\in\wedge^{l+1}(V)$ (or $\wedge^{l+1}(V,V)$)
one defines the \emph{interior product $i_\phi\psi\in\wedge^{k+l}(V)$} (resp. $\wedge^{l+1}(V,V)$) as
\begin{align}
  i_\phi\psi(x_1,\ldots,x_{k+l}) = \frac{1}{k!l!}
  \sum_{\sigma\in S_{k+l}}\psi\paraa{\phi(x_{\sigma(1)},\ldots,x_{\sigma(k)}),x_{\sigma(k+1)},\ldots,x_{\sigma(k+l)}}.
\end{align}
The generalized Jacobi identity for $\phi$ can then be expressed as
$i_\phi\phi=0$. The interior product satisfies the following
properties with respect to the wedge product (see e.g. \cite{m:remnijenhuis})
\begin{align}
  &i_{\tau\wedge\phi}\psi = \tau\wedge(i_{\phi}\psi)\\
  &i_\phi(\tau\wedge\psi) = (i_\phi\tau)\wedge\psi+(-1)^{(k-1)p}\tau\wedge(i_\phi\psi),
\end{align}
where $\tau\in\Lambda^p(V)$. With these set of relations at hand, one
can now easily prove the following statement.

\begin{proposition}\label{prop:tauwedgephiGJI}
  If $\phi\in\wedge^{n}(V,V)$ satisfies the generalized Jacobi
  identity and $\tau$ is a $\phi$-compatible $p$-form, then
  $\tau\wedge\phi$ satisfies the generalized Jacobi identity.
\end{proposition}

\begin{proof}
  The generalized Jacobi identity can be expressed as
  $i_{\tau\wedge\phi}(\tau\wedge\phi)=0$. One computes
  \begin{align*}
    i_{\tau\wedge\phi}(\tau\wedge\phi) &= \tau\wedge\paraa{i_\phi(\tau\wedge\phi)}
    = \tau\wedge\paraa{(i_\phi\tau)\wedge\phi+(-1)^{(n-1)p}\tau\wedge(i_\phi\phi)}\\
    &= \tau\wedge(i_\phi\tau)\wedge\phi = 0,
  \end{align*}
  since $\tau$ is $\phi$-compatible (which implies $i_\phi\tau=0$) and
  $i_\phi\phi=0$ by assumption.
\end{proof}

\noindent It is known that the bracket of any $n$-Lie algebra
satisfies the generalized Jacobi identity (see
e.g. \cite{aip:genpoissonstruct}), which can be shown by symmetrizing
the fundamental identity of the $n$-Lie algebra. Starting from an $n$-Lie
algebra, Proposition \ref{prop:tauwedgephiGJI} does not guarantee that
$\tau\wedge\phi$ defines a Nambu-Lie algebra. However, if we assume
that
\begin{align*}
  \parab{\paraa{i_{x_1\cdots x_{p-1}}\tau}\wedge\tau}(y_1,\ldots,y_{p+1})
  = \paraa{\tau(x_1,\ldots,x_{p-1},\cdot)\wedge\tau}(y_1,\ldots,y_{p+1}) = 0
\end{align*}
for all $x_1,\ldots,x_{p-1},y_1,\ldots,y_{p+1}\in V$, then the desired result follows.

\begin{proposition}\label{prop:tauwedgephiNambuLie}
  Let $(V,\phi)$ be an $n$-Nambu-Lie algebra and assume that $\tau$ is
  a $\phi$-compatible $p$-form such that $(i_{x_1\cdots
    x_{p-1}}\tau)\wedge\tau=0$ for all $x_1,\ldots,x_{p-1}\in V$. Then
  $(V,\tau\wedge\phi)$ is a $(n+p)$-Nambu-Lie algebra.
\end{proposition}

\begin{proof}
  Let us start by introducing some notation. Let
  $\X=(x_1,\ldots,x_{m-1})$ and $\Xt=(x_1,\ldots,x_m)$ be ordered
  sets, and let $X$ be an ordered subset of $\Xt$ such that
  $X=(x_{i_1},\ldots,x_{i_p})$ with $i_k<i_l$ if $k<l$. By
  $\Xt\backslash X$ (or $\X\backslash X$) we mean the ordered set
  obtained from $\Xt$ (or $\X$) by removing the elements in
  $X$. Similarly, we set $\Y=(y_1,\ldots,y_m)$ and let $Y$ be an
  ordered subset of $\Y$.

  Let $x_1,\ldots,x_{m-1},y_1,\ldots,y_m$ be arbitrary elements of $V$
  and set $$x_m=\paraa{\tau\wedge\phi}(y_1,\ldots,y_m).$$ 
  The fundamental identity can then be written as
  \begin{equation}\label{eq:NJIproof}
    \begin{split}
      \text{FI} &=
      \sum_{k=1}^m\paraa{\tau\wedge\phi}(y_1,\ldots,y_{k-1},(\tau\wedge\phi)
      (x_1,\ldots,x_{m-1},y_k),y_{k+1},\ldots,y_m)\\
      &\qquad\qquad-\paraa{\tau\wedge\phi}(x_1,\ldots,x_m)=0.
    \end{split}
  \end{equation}
  Since $\tau$ is compatible with $\phi$, all of the terms with $x_m$
  inside $\tau$ vanish, and one can rewrite the second term in
  (\ref{eq:NJIproof}) as
  \begin{align}\label{eq:NJIFirstXY}
    \paraa{\tau\wedge\phi}(x_1,\ldots,x_m)
    = \sum_{X\subset\X,Y\subset\Y}\sgn(X)\sgn(Y)\tau(X)\tau(Y)
    \phi\paraa{\X\backslash X,\phi(\Y\backslash Y)},
  \end{align}
  where $\sgn(X)$ denotes the sign of the permutation $\sigma\in S_{m}$ such
  that
  \begin{align*}
    X &= (x_{\sigma(1)},\ldots,x_{\sigma(p)})\\
    (\X\backslash X,x_m) &= (x_{\sigma(p+1)},\ldots,x_{\sigma(m)}),
  \end{align*}
  and $\sgn(Y)$ denotes the sign of the permutation $\rho\in S_m$ such
  that
  \begin{align*}
    Y &= (y_{\rho(1)},\ldots,y_{\rho(p)})\\
    \Y\backslash Y &= (y_{\rho(p+1)},\ldots,y_{\rho(m)}).
  \end{align*}
  Let us now turn to the first term in (\ref{eq:NJIproof}), which will
  generate two types of terms. The first kind, $A$, will include a
  $\tau$ acting on $y_k$, and the second kind, $B$, include no such
  $\tau$. One can now rewrite $B$ in a fashion similar to
  (\ref{eq:NJIFirstXY})
  \begin{align*}
    B &= \sum_{X\subset\X,Y\subset\Y}\sgn(X)\sgn(Y)\tau(X)\tau(Y)\times\\
    &\qquad\sum_{y_{i_k}\in\Y\backslash Y}
    \phi\paraa{y_{i_1},\ldots,y_{i_{k-1}},\phi(\X\backslash X,y_{i_k}),y_{i_{k+1}},\ldots,y_{i_n}},
  \end{align*}
  where $(y_{i_1},\ldots,y_{i_n})=\Y\backslash Y$. Now, one notes that
  subtracting $B$ from (\ref{eq:NJIFirstXY}) gives zero due to the
  fact that $\phi$ satisfies the fundamental identity. Thus, there
  will only be terms of type $A$ left in (\ref{eq:NJIproof}). To
  rewrite these terms we introduce yet some notation.

  Let $\Yt=(y_{i_1},\ldots,y_{i_{n-1}})$ be an ordered subset of $Y$
  and let $\Xb$ be an ordered subset of $\X$ with $|\Xb|=n$. By
  $\sgn(\Yt_k)$ we denote the sign of the permutation $\sigma$ such
  that $\sigma(m)=k$ and
  \begin{align*}
    \Yt &= (y_{\sigma(p+1)},\ldots,y_{\sigma(m-1)})\\
    \Y\backslash\Yt &= (y_{\sigma(1)},\ldots,y_{\sigma(p)}),
  \end{align*}
  and by $\sgn(\Xb)$ we denote the sign of the permutation $\rho$ such
  that $\rho(p)=m$ and
  \begin{align*}
    \Xb &= (x_{\rho(p+1)},\ldots,x_{\rho(m)})\\
    \X\backslash\Xb &= (x_{\rho(1)},\ldots,x_{\rho(p-1)})
  \end{align*}
  with the definition $x_m= y_k$. In this notation, the terms of type
  $A$ can be written as
  \begin{align*}
    A &= \sum_{\Yt\subset\Y,\Xb\subset\X}\!\!\!\!
    \sgn(\Xb)\phi\paraa{\Yt,\phi(\Xb)}\sum_{y_k\in\Y\backslash\Yt}\sgn(\Yt_k)
    \tau\paraa{\X\backslash\Xb,y_k}\tau\paraa{\Y\backslash\Yt\backslash\{y_k\}}\\
    &\propto\sum_{\Yt\subset\Y,\Xb\subset\X}\!\!\!\!
    \sgn(\Xb)\phi\paraa{\Yt,\phi(\Xb)}
    \paraa{(i_{x_{\rho(1)}\cdots x_{\rho(p-1)}}\tau)\wedge\tau}(\Y\backslash\Yt) = 0,
  \end{align*}
  since $(i_{x_1\cdots x_{p-1}}\tau)\wedge\tau=0$ for all $x_1,\ldots,x_{p-1}\in V$.
\end{proof}

\noindent The following example provides a generic construction in the
case when $\phi$ maps $V^n$ onto a proper subspace of $V$.

\begin{example}
  Let $(V,\phi)$ be an $n$-Lie algebra, with $\dim(V)=m$, such
  that $\phi:V^n\to U$, where $U$ is a $m-p$ dimensional subspace of
  $V$ ($p\geq 1$). Given a basis $u_1,\ldots,u_{m-p}$ of $U$, we define
  $\tau\in\bigwedge^pV$ as
  \begin{align*}
    \tau(v_1,\ldots,v_p) = \det(v_1,\ldots,v_p,u_1,\ldots,u_{m-p}).
  \end{align*}
  Then $\tau$ is a $\phi$-compatible $p$-form on $V$. Let us now show
  that $(i_{v_1\cdots v_{p-1}}\tau)\wedge\tau=0$. From the definition
  of $\tau$ one obtains
  \begin{align*}
    &\paraa{(i_{v_1\cdots v_{p-1}}\tau)\wedge\tau}(w_1,\ldots,w_{p+1})\\
    &\qquad=\sum_{\sigma\in S_{p+1}}\sgn(\sigma)
    \det(v_1,\ldots,v_{p-1},w_{\sigma(1)},u_1,\ldots,u_{m-p})\\
    &\qquad\qquad\times\det(w_{\sigma(2)},\ldots,w_{\sigma(p+1)},u_1,\ldots,u_{m-p}).
  \end{align*}
  For this to be non-zero, one needs first of all that
  $v_1,\ldots,v_{p-1},u_1,\ldots,u_{m-p}$ and $w_1,\ldots,w_{p+1}$ are
  two sets of linearly independent vectors. Even though this is is the
  case every term in the sum is zero since for a non-zero result one
  needs that $w_1,\ldots,w_{p+1}$ are linearly independent of
  $u_1,\ldots,u_{m-p}$, which is impossible due to the fact that
  $w_1,\ldots,w_{p+1}$ are linearly independent and
  $\dim(V)=m$. Hence, by Proposition \ref{prop:tauwedgephiNambuLie},
  $(V,\tau\wedge\phi)$ is a $(n+p)$-Nambu-Lie algebra.
\end{example}


\begin{thebibliography}{dAIPB97}

\bibitem[dAIPB97]{aip:genpoissonstruct}
J.~A. de~Azc{\'a}rraga, J.~M. Izquierdo, and J.~C. P{\'e}rez~Bueno.
\newblock On the higher-order generalizations of {P}oisson structures.
\newblock {\em J. Phys. A}, 30(18):L607--L616, 1997.

\bibitem[Mic87]{m:remnijenhuis}
Peter~W. Michor.
\newblock Remarks on the {F}r\"olicher-{N}ijenhuis bracket.
\newblock In {\em Differential geometry and its applications ({B}rno, 1986)},
  volume~27 of {\em Math. Appl. (East European Ser.)}, pages 197--220. Reidel,
  Dordrecht, 1987.

\bibitem[Yau10]{y:nhomliegoingdown}
D.~Yau.
\newblock On $n$-ary {H}om-{N}ambu and {H}om-{N}ambu-{L}ie algebras.
\newblock \texttt{arXiv:1004.2080}, 2010.

\end{thebibliography}


\begin{thebibliography}{99}

\bibitem{Kerner7}  Abramov V., Le Roy B.,  and Kerner R., Hypersymmetry: a $\mathbb{Z}_3$-graded generalization of supersymmetry,  J. Math. Phys., \textbf{38} (3), 1650--1669 (1997).

\bibitem{AizawaSaito}
Aizawa, N., Sato, H., $q$-deformation of the Virasoro algebra with central extension, Phys. Lett. B
\textbf{256}, no. 1, 185--190 (1991). Hiroshima University preprint
preprint HUPD-9012 (1990).

\bibitem{abhhs:fuzzy}
Arnlind J., Bordemann M., Hofer L., Hoppe J., and Shimada H.,
Fuzzy Riemann surfaces, J. High Energy Phys., no. 6, 047, 18 pp. (2009). 
\texttt{hep-th/0602290}.

\bibitem{abhhs:noncommutative}
Arnlind J., Bordemann M., Hofer L., Hoppe J., and Shimada H.,
Noncommutative {R}iemann surfaces by embeddings in {$\Bbb R\sp 3$}.
{\em Comm. Math. Phys.}, \textbf{288} (2): 403--429 (2009).

\bibitem{a:phdthesis}
Arnlind J., {\em Graph Techniques for Matrix Equations and Eigenvalue Dynamics}.
PhD thesis, Royal Institute of Technology, 2008.

\bibitem{a:repcalg}
Arnlind J., Representation theory of {$C$}-algebras for a higher-order class of
spheres and tori, {\em J. Math. Phys.}, \textbf{49} (5): 053502, 13 (2008).

\bibitem{as:affinecrossed}
Arnlind J., Silvestrov S.,
Affine transformation crossed product type algebras and noncommutative surfaces
{\em Contemp. Math. } \textbf{503}, 2009.

\bibitem{ahh:multilinear} Arnlind J., Hoppe J., Huisken G., Multi linear formulation of differential geometry and matrix regularizations. \texttt{arXiv:1009.4779} (2010).

\bibitem{ams:ternary} Arnlind J., Makhlouf A., Silvestrov S., 
Ternary Hom-Nambu-Lie algebras induced by Hom-Lie algebras, J. Math. Phys. \textbf{51}, 043515, 11 pp. (2010)

\bibitem{ams:gennambu}  Ataguema H.,  Makhlouf A., Silvestrov S., Generalization of n-ary Nambu
algebras and beyond. J. Math. Phys. 50, 083501, (2009). (DOI:10.1063/1.3167801)

\bibitem{amdy:quantnambu} Awata H., Li M., Minic D., Yoneya T., On the quantization of Nambu brackets,
J. High Energy Phys. \textbf{2}, Paper 13, 17 pp. (2001)

\bibitem{aip:genpoissonstruct}
 de~Azc{\'a}rraga  J.~A.,  Izquierdo J.~M., and P{\'e}rez~Bueno  J.~C., 
On the higher-order generalizations of {P}oisson structures, {\em J. Phys. A}, \textbf{30} (18): L607--L616 1997.

\bibitem{aip:reveiw}
de~Azc{\'a}rraga  J.~A.,  Izquierdo J.~M.,  n-ary algebras: a review with applications, J. Phys. A43 (2010) 293001-1-117.

\bibitem{BL2007}  Bagger J., Lambert N., Gauge Symmetry and
Supersymmetry of Multiple M2-Branes, \texttt{arXiv:0711.0955} (2007).

\bibitem{IGSC}  Caenepeel S.,  Goyvaerts I., Hom-Hopf algebras. Preprint arXiv: 0907.0187, (2009)

\bibitem{CassasLodayPirashvili}  Cassas J. M., Loday J.-L. and Pirashvili T.,  Leibniz
 $n$-algebras, Forum Math. \textbf{14}, 189--207 (2002).


\bibitem{ChaiElinPop} Chaichian M., Ellinas D. and
Z. Popowicz, Quantum conformal algebra with
central extension, Phys. Lett. B \textbf{248}, no. 1-2, 95--99 (1990).

\bibitem{ChaiKuLukPopPresn} Chaichian M.,
Isaev A. P.,  Lukierski J., Popowic Z. and
Pre\v{s}najder P., $q$-deformations of Virasoro
algebra and conformal dimensions, Phys. Lett. B
\textbf{262} (1), 32--38 (1991).

\bibitem{ChaiKuLuk}  Chaichian M., Kulish P. and
Lukierski J.,  $q$-deformed Jacobi identity, $q$-oscillators and $q$-deformed
infinite-dimensional algebras, Phys. Lett. B
\textbf{237} , no. 3-4, 401--406 (1990).

\bibitem{ChaiPopPres} Chaichian M.,
Popowicz Z., Pre\v{s}najder P.,  $q$-Virasoro
algebra and its relation to the $q$-deformed KdV
system, Phys. Lett. B \textbf{249}, no. 1,
63--65 (1990).

\bibitem{CurtrZachos1} Curtright T. L.,  Zachos C. K.,
Deforming maps for quantum algebras, Phys. Lett.
B \textbf{243}, no. 3, 237--244, (1990).

\bibitem{DaskaloyannisGendefVir}
Daskaloyannis C., Generalized deformed Virasoro
algebras, Modern Phys. Lett. A \textbf{7} no. 9, 809--816 (1992)

\bibitem{f:nliealgebras} Filippov V.~T., n-Lie algebras,
(Russian), Sibirsk. Mat. Zh. \textbf{26}, no. 6, 126--140 (1985).
(English translation: Siberian Math. J. \textbf{26}, no. 6, 879--891 (1985))

\bibitem{HLS} Hartwig J. T., Larsson D., Silvestrov S. D.,
Deformations of Lie algebras using $\sigma$-derivations, J.
of Algebra \textbf{295},  314--361 (2006). Preprint
in Mathematical Sciences 2003:32, LUTFMA-5036-2003, Centre for
Mathematical Sciences, Department of Mathematics, Lund Institute
of Technology, 52 pp. (2003).

\bibitem{HoppeJMalgNambumech} Hoppe J.,
On $M$-algebras, the quantisation of Nambu-mechanics, and volume preserving diffeomorphisms,
Helv. Phys. Acta \textbf{70}, no. 1-2, 302--317 (1997). arXiv:hep-th/9602020v1

\bibitem{Hu}  Hu N., $q$-Witt algebras,
$q$-Lie algebras, $q$-holomorph structure and
representations,  Algebra Colloq. \textbf{6},
no. 1, 51--70 (1999).

\bibitem{Kerner} Kerner R., Ternary algebraic structures and their
applications in physics, in the ``Proc. BTLP 23rd International
Colloquium on Group Theoretical Methods in Physics'', (2000).
arXiv:math-ph/0011023v1

\bibitem{Kerner2} Kerner R., $\mathbb{Z}_3$-graded algebras and non-commutative gauge theories,
dans le livre "Spinors, Twistors, Clifford Algebras and Quantum
Deformations", Eds. Z. Oziewicz, B. Jancewicz, A. Borowiec, pp.
349--357, Kluwer Academic Publishers (1993).

\bibitem{Kerner4} Kerner R., The cubic chessboard: Geometry and physics, Classical Quantum Gravity \textbf{14}, A203-A225 (1997)

\bibitem{Kerner6} Kerner R., Vainerman L., On special classes of
$n$-algebras, J. Math. Phys., \textbf{37} (5), 2553--2565 (1996).

\bibitem{LS1} Larsson D., Silvestrov S. D., Quasi-Hom-Lie algebras, Central Extensions and $2$-cocycle-like
identities. J. Algebra \textbf{288}, 321--344 (2005).
Preprints in Mathematical Sciences  2004:3, LUTFMA-5038-2004, Centre for Mathematical Sciences, Department of Mathematics, Lund Institute of Technology, Lund University, 2004.

\bibitem{LS2} Larsson D.,  Silvestrov S. D., Quasi-Lie algebras. In
"Noncommutative Geometry and Representation Theory in Mathematical Physics". Contemp. Math., 391, Amer. Math. Soc., Providence, RI, (2005), 241--248. Preprints in Mathematical Sciences 2004:30, LUTFMA-5049-2004, Centre for Mathematical Sciences, Department of Mathematics, Lund Institute of Technology, Lund University, 2004.

\bibitem{LS3}
Larsson D., Silvestrov S. D., Quasi-deformations of $sl_2(\mathbb{F})$ using twisted derivations, 
Comm. in Algebra \textbf{35} , 4303 -- 4318 (2007).

\bibitem{LiuKQuantumCentExt} Liu, K. Q., Quantum central extensions,
C. R. Math. Rep. Acad. Sci. Canada \textbf{13} (4), 135--140 (1991).

\bibitem{LiuKQCharQuantWittAlg} Liu K. Q., Characterizations of the
Quantum Witt Algebra, Lett. Math. Phys. \textbf{24} (4), 257--265 (1992)

\bibitem{LiuKQPhDthesis} Liu K. Q., The Quantum Witt Algebra and
Quantization of Some Modules over Witt Algebra, PhD Thesis, Department of Mathematics, University of Alberta, Edmonton, Canada (1992)

\bibitem{MS} Makhlouf A., Silvestrov S. D., Hom-algebra structures.
J. Gen. Lie Theory Appl. Vol \textbf{2} (2), 51--64 (2008).
Preprints in Mathematical Sciences  2006:10, LUTFMA-5074-2006, Centre for Mathematical Sciences, Department of Mathematics, Lund Institute of Technology, Lund University, 2006.

\bibitem{HomHopf} Makhlouf A., Silvestrov S. D., Hom-Lie admissible Hom-coalgebras and Hom-Hopf algebras. In "Generalized Lie
theory in Mathematics, Physics and Beyond.
S. Silvestrov, E. Paal, V. Abramov, A. Stolin, Editors". Springer-Verlag, Berlin,
Heidelberg, Chapter 17, 189--206, (2009).
Preprints in Mathematical Sciences, Lund University, Centre for Mathematical Sciences, Centrum Scientiarum Mathematicarum (2007:25) LUTFMA-5091-2007 and in arXiv:0709.2413 [math.RA] (2007)

\bibitem{HomAlgHomCoalg} Makhlouf A., Silvestrov S. D., Hom-Algebras and Hom-Coalgebras. Journal of Algebra and its Applications, Vol. \textbf{9}, (2010).
Preprints in Mathematical Sciences, Lund University, Centre for Mathematical Sciences, Centrum Scientiarum Mathematicarum, (2008:19) LUTFMA-5103-2008 and in arXiv:0811.0400 [math.RA] (2008). 

\bibitem{HomDeform} Makhlouf A., Silvestrov S. D., Notes on 1-parameter formal deformations of Hom-associative and Hom-Lie algebras, 
Forum Math. 22 (2010), no. 4, 715--739.
Preprints in Mathematical Sciences, Lund University, Centre for Mathematical Sciences, Centrum Scientiarum Mathematicarum, (2007:31) LUTFMA-5095-2007. arXiv:0712.3130v1 [math.RA] (2007).

\bibitem{m:remnijenhuis}
Michor P. W., Remarks on the {F}r\"olicher-{N}ijenhuis bracket, 
In {\em Differential geometry and its applications ({B}rno, 1986)},
volume~27 of {\em Math. Appl. (East European Ser.)}, pages 197--220. Reidel,
Dordrecht, 1987.

\bibitem{n:generalizedmech} Nambu Y., Generalized Hamiltonian
mechanics, Phys. Rev. D (3), \textbf{7}, 2405--2412 (1973)

\bibitem{Takhtajan} Takhtajan L., On foundation of the generalized
Nambu mechanics, Comm. Math. Phys. \textbf{160}, 295--315 (1994)

\bibitem{Yau:EnvLieAlg} Yau D.,
Enveloping algebra of Hom-Lie algebras, J. Gen.
Lie Theory Appl. \textbf{2} (2), 95--108 (2008)

\bibitem{Yau:HomolHomLie} Yau D.,
Hom-algebras  and homology,
 J. Lie Theory \textbf{19} (2009) 409--421.

\bibitem{Yau:HomBial} Yau D.,
Hom-bialgebras and comodule algebras,
International Electronic Journal of Algebra\textbf{ 8} (2010) 45-64.



\bibitem{y:nhomliegoingdown}
Yau D., On $n$-ary {H}om-{N}ambu and {H}om-{N}ambu-{L}ie algebras, \texttt{arXiv:1004.2080}, 2010.

\bibitem{y:nhomliehommalcev}
Yau D., On n-ary Hom-Nambu and Hom-Maltsev algebras,	arXiv:1004.4795, 2010.

\bibitem{y:nhomliehomass}
Yau D., A Hom-associative analogue of n-ary Hom-Nambu algebras, 	arXiv:1005.2373, 2010.

\end{thebibliography}
\end{document}